\documentclass[amssymb,reqno,amsfonts,refcheck,12pt,verbatim,righttag]{amsart}

\usepackage{graphicx}
\usepackage{graphics,color}
\usepackage{amssymb,mathrsfs}
\usepackage{graphicx,color,tikz,caption,subcaption}
\usepackage[all]{xy}
\usepackage{mathtools}
\usepackage{enumerate}
\usepackage{amsmath}
\usepackage{bbm}
\usepackage{bm}
\usepackage{pdfsync}

\usepackage[colorlinks, linkcolor=blue,citecolor=blue, anchorcolor=blue,urlcolor=blue,pagebackref,hypertexnames=false]{hyperref}

\numberwithin{equation}{section} 
\newtheorem{thm}{Theorem}[section]
\newtheorem{lem}[thm]{Lemma}
\newtheorem{prop}[thm]{Proposition}

\newtheorem{rem}[thm]{Remark}

\newcommand{\R}{\mathbb{R}}
\newcommand{\C}{\mathbb{C}}
\newcommand{\Q}{\mathbb{Q}}
\newcommand{\N}{\mathbb{N}}
\newcommand{\Z}{\mathbb{Z}}

\usepackage{float}

\begin{document}
\baselineskip 14pt

\title{Gaussian rational numbers in Cantor sets in the complex plane}

\author{Yu-Feng Wu}
\address[]{School of Mathematics and Statistics\\HNP-LAMA\\Central South University\\ Changsha, 410083, PR China}
\email{yufengwu.wu@gmail.com}



\date{}

\begin{abstract}
Given $\beta\in\Z[i]$ with $|\beta|>1$ and a finite set $D\subset \Q(i)$, let 
\[K_{\beta, D}=\left\{\sum_{j=1}^{\infty}\frac{d_j}{\beta^j}: d_j\in D, \forall j\geq 1\right\}.\] 
Let $\mathcal{S}$ be a finite set of non-associate prime elements in $\Z[i]$ not dividing $\beta$. We prove that if the Hausdorff dimension of $K_{\beta,D}$ is less than $1$, then there are only finitely many Gaussian rational numbers in $K_{\beta,D}$ whose  denominators have  all their prime factors in $\mathcal{S}$.
\end{abstract}

\maketitle

\section{Introduction}\label{S-1}

In this paper, we study the distribution of Gaussian rational numbers in a class of self-similar sets in the complex plane. We begin by reviewing related studies on the real line. The middle-third Cantor $\bm{C}$ consists of all real numbers in $[0,1]$ whose ternary expansions do not have the digit $1$. That is,
\begin{equation}\label{eqMiddleCantor}
\bm{C}=\left\{\sum_{j=1}^{\infty}\frac{a_j}{3^j}: a_j\in \{0,2\}, j\geq 1\right\}.
\end{equation}
In 1984,  Mahler \cite{Mahler84} proposed to study the problem of how well  elements of $\bm{C}$ can  be
approximated by rational numbers in $\bm{C}$ and by rational numbers outside  $\bm{C}$. Mahler's problem pioneered the study of Diophantine approximation on fractals, leading to a substantial amount of research; see, for example, \cite{LSV07O,BD16M,Shparlinski21,TWW24M,Baker25,CVY24C,CU25S,BHZ24K}.

Closely related to Mahler's problem is the question of what  the rational numbers in the middle-third Cantor set are. Wall \cite{Wall83} first noted that $0,\frac{1}{4}, \frac{3}{4}$ and $1$ are the only dyadic rationals in $\bm{C}$. He also showed in \cite{Wall90} that there are  exactly $14$  rationals in $\bm{C}$ with decimal denominators. Later, Nagy \cite{Nagy01} proved that for any prime number $p>3$, there are only finitely many rational numbers in $\bm{C}$ whose denominators are powers of $p$.  Bloshchitsyn \cite{Bloshchitsyn15} obtained analogous results for the class of  $m$-adic Cantor sets, a generalization of $\bm{C}$ obtained by replacing the base $3$ with an integer $m\geq 3$ and the digit set $\{0,2\}$ with a proper subset of $\{0,1,\ldots,m-1\}$ in \eqref{eqMiddleCantor}. Recently, Schleischitz \cite{Schleischitz21} proved a more general result in this setting: if $P$ is a finite set of distinct primes not dividing $m$, then there are only finitely many rational numbers in such a set whose denominators have all prime factors from $P$. Shparlinski \cite{Shparlinski21} later obtained a quantitative version of this result using a different method. Li et al. \cite{LLW23} extended Shparlinski's result to the broader setting of $\times m$-invariant sets on $[0,1]$. Jiang  et al. \cite{JKLW24R} proved a finiteness result for translations of $m$-adic Cantor sets. Very  recently, Kong et al. \cite{KLW25} obtained a further extension of Schleischitz's result by allowing the digit set to be a finite subset of $\Q$. 

Solomon, Trauthwein and Weiss \cite{RSTW20} studied the problem of counting rational numbers in $\bm{C}$ up to a given height. They  formulated a conjecture asserting that for any $\epsilon>0$, there are $\ll N^{\log_32+\epsilon}$ reduced rational numbers in $\bm{C}$ with denominators at most $N$. This conjecture (and its natural generalizations) remains open; see \cite{CVY24C,CU25S} for significant  progress. Here and throughout, the notation $a(x)\ll b(x)$ means that $a(x)\leq cb(x)$ for some positive constant $c$ independent of $x$. 

The purpose of this paper is to generalize some of the above results to Gaussian rational numbers in a class of  self-similar sets in the complex plane. To state our setting and result, we first introduce some background on self-similar sets.

A mapping $\phi:\R^d\to \R^d$ is called {\em contracting} if there exists $c\in (0,1)$ such that 
\[|\phi(x)-\phi(y)|\leq c|x-y|, \quad \forall x, y\in\R^d.\]
If equality holds, we call $\phi$ a {\em similitude} and $c$ the {\em similarity ratio} of $\phi$. 
We call a finite collection $\Phi=\{\phi_i\}_{i=1}^{\ell}$ of contracting mappings from $\R^d$ to itself an {\em iterated function system (IFS)}. According to a classical result of Hutchinson \cite{hutchinson}, given an IFS $\Phi=\{\phi_i\}_{i=1}^{\ell}$ on $\R^d$, there is a unique non-empty compact set $E\subset\R^d$, called the {\em attractor} of $\Phi$, such that $$E=\bigcup_{i=1}^{\ell}\phi_i(E).$$ 
We say that $E$ is generated by $\Phi$, and call $\Phi$ a {\em generating IFS} for $E$.  
The attractor $E$ is called a {\em self-similar set} if all $\phi_i$ are similitudes on $\R^d$. In this case, if all $\phi_i$ have the same linear part, we call $\Phi$ an {\em homogeneous} IFS. For instance, the middle-third Cantor set defined in \eqref{eqMiddleCantor} is a self-similar set generated by the IFS $\{\frac{x}{3},\frac{x+2}{3}\}$. Self-similar sets in $\C$ are naturally understood to be the self-similar sets on $\R^2$. 

For a set $A\subset \R^d$, we denote by $\dim_{\rm H}A$ the Hausdorff dimension of $A$; see \cite{Falconer14} for the definition. Given an IFS  $\Phi=\{\phi_i\}_{i=1}^{\ell}$ on $\R^d$ which generates a self-similar set $E$,  we define the {\em similarity dimension} of $\Phi$ to be the positive number $s$ satisfying
\[\sum_{i=1}^{\ell}\rho_{i}^s=1,\]
where $\rho_i\in (0,1)$ is the similarity ratio of $\phi_i$, $i=1,\ldots,\ell$. 
It is well-known that the similarity dimension of $\Phi$ is always an upper bound for the Hausdorff dimension of $E$. Equality holds under certain separation conditions, for instance,  the open set condition; see e.g. \cite{Falconer14}.

We also need some background on the Gaussian integers and Gaussian rational numbers, together with some more notation.  Throughout, let $\N$ be the set of positive integers and $\N_0:=\N\cup\{0\}$.
Let $\mathbb{Z}[\mathrm{i}] = \{a+b\mathrm{i} : a, b\in \mathbb{Z}\}$ be the ring of Gaussian integers.  The units of $\mathbb{Z}[\mathrm{i}]$ are $\pm 1$ and $\pm \mathrm{i}$. Let $\Q(i)=\{\omega/\gamma: \omega,\gamma\in\Z[i], \gamma\neq 0\}$ be the set of Gaussian rational numbers. For $\alpha, \beta \in \mathbb{Z}[\mathrm{i}]$, we say that $\alpha$ is a {\em divisor} (or a {\em factor}) of $\beta$, or say $\alpha$ {\em divides} $\beta$ and write as $\alpha \mid \beta$, if  $\beta=\alpha\gamma$ for some $\gamma\in\Z[i]$. We say that $\alpha$ is an {\em associate} of $\beta$ if $\alpha=u\beta$, where $u$ is a unit of $\Z[i]$. We call  $\beta \in \mathbb{Z}[\mathrm{i}] \setminus \{0, \pm 1, \pm \mathrm{i}\}$ a {\em  prime element} if  its only divisors are the units and its own associates.  Prime elements $\beta_1,\ldots,\beta_k\in\Z[i]$ are said to be {\em non-associate} if for any $i\neq j$, $\beta_i$ is not an associate of $\beta_j$. For $\alpha\in \C$, $|\alpha|$ is the modulus of $\alpha$ and $\bar{\alpha}$ is the conjugate of $\alpha$.

In this paper, we are interested in  self-similar sets in $\C$ generated by homogeneous IFSs with the common linear part being the reciprocal of a Gaussian integer and the translational parts being Gaussian rational numbers. More precisely, let $\beta\in \Z[i]$ with $|\beta|>1$ and $D=\{t_1,\ldots,t_{\ell}\}\subset \Q(i)$ be a finite set. Let $K_{\beta,D}$ be the self-similar set in $\C$ generated by the IFS
\begin{equation}\label{eqIFS}
\Phi_{\beta,D}:=\left\{\phi_j(z)=\frac{z+t_j}{\beta}: 1\leq j\leq \ell\right\}.
\end{equation}
It is easy to see that $K_{\beta,D}$ can also be expressed as 
\begin{equation}\label{eqKbetaDdef}
K_{\beta, D}=\left\{\sum_{j=1}^{\infty}\frac{d_j}{\beta^j}: d_j\in D, \forall j\geq 1\right\}.
\end{equation}
The self-similar set $K_{\beta,D}$ (especially some of its subclasses) has been studied in several mathematical contexts. For instance, when $\beta=-n\pm i$ and $D=\{0,1,\ldots, n^2-1\}$ with $n\in\N$, $K_{\beta,D}$ appears in the study of number systems with complex bases \cite{KS75C} and self-similar tilings of the plane \cite{Gilbert86T}; very recently, Diophantine approximation on $K_{\beta,D}$  has been studied  in \cite{RHS23C} when $D$ is a proper subset of $\{0,1,\ldots,n^2-1\}$. For $\beta\in\C$ with $|\beta|>1$, there have been works on the topology of $K_{\beta,\{-1,1\}}$ \cite{SX03O}, as well as self-similar measures on $K_{\beta,\{-1,1\}}$ and its generalizations \cite{SX03O,SS16A,MO23F}. Very recently, the pointwise normality of self-similar measures on $K_{\beta,D}$ has been investigated in  \cite{LWXZ25E} when $D\subset\Z[i]$. In this paper, we contribute to the study of $K_{\beta,D}$ by proving a result on Gaussian rational numbers in $K_{\beta,D}$, in a spirit  similar to the results on rational numbers in Cantor sets on $\R$ surveyed above.

The main result of this paper is the following.
\begin{thm}\label{thmfinite}
Let $\beta\in \Z[i]$ with $|\beta|>1$ and $D\subset \Q(i)$ be a finite set. Let $K_{\beta,D}$ be the self-similar set defined as in \eqref{eqKbetaDdef}.
Let  $\mathcal{S}$ be a finite
set of non-associate prime elements in $\Z[i]$ not dividing $\beta$. Suppose  that $\dim_{\rm H}K_{\beta,D}<1$. 
Then there are only finitely many Gaussian rational numbers
$\frac{\omega}{\gamma}\in K_{\beta,D}$ with $\gamma\in \Z[i]$ whose prime factors all belong to $\mathcal{S}$.
\end{thm}

\begin{rem}
In Theorem \ref{thmfinite}, the assumption that $\dim_{\rm H}K_{\beta, D}<1$ cannot in general be  relaxed. For instance, the unit interval $[0,1]$ (viewed as a subset of $\C$) is the self-similar set $K_{2,\{0,1\}}$, which has Hausdorff dimension $1$ and clearly contains infinitely many (Gaussian) rationals whose denominators are powers of $3$. 
\end{rem}

For the proof of Theorem \ref{thmfinite}, besides adopting and extending some
ideas and strategies from the previous works \cite{Bloshchitsyn15,Schleischitz21,KLW25} on Cantor sets on $\R$, we also need to develop new techniques to overcome the difficulties caused by the IFS as well as by the structure of $\Z[i]$. Below we briefly outline the strategy employed in our proof of Theorem \ref{thmfinite}. 

To prove Theorem \ref{thmfinite}, we first establish a counting result on the number of rational vectors (with the same denominator in components) up to a given height in self-similar sets on $\R^d$ (see Theorem \ref{thmCounting}), which slightly improves a result of Schleischitz \cite{Schleischitz21}. We achieve this by first proving a rough counting result (see Proposition \ref{thmdimB}) for general bounded subsets of $\R^d$ in terms of the upper box-counting dimension, using an idea from  \cite{Schleischitz21}. We  then show that this can be refined to obtain Theorem~\ref{thmCounting} with the aid of a simple property of self-similar sets. For our purpose, we then formulate a version of the counting result for Gaussian rational numbers in self-similar sets in $\C$ (see Theorem \ref{thmAS}). In this step, we introduce the denominator height $H(\cdot)$ for Gaussian integers; see \eqref{defReducedNorm}. It turns out that $H$ is not only suitable for measuring the height of Gaussian rational numbers (as in Theorem \ref{thmAS}), but also plays a crucial role in the subsequent development. Indeed, our next step is to use the counting result (Theorem \ref{thmAS}) to show that every $\frac{\omega}{\gamma}\in  K_{\beta,D}\cap \Q(i)$ admits an eventually periodic coding with period $n\ll H(\gamma)^s$, where $s$ can be made less than $1$ due to the assumption that $\dim_{\rm H}K_{\beta,D}<1$; see Theorem \ref{lemperiod}. On the other hand, by establishing via a series of lemmas the properties of $H$, we are able to prove a result on the multiplicative order for Gaussian integers (see Lemma \ref{lemorderlowerEs}). Then, in the final step, we use this and a trick from \cite{KLW25} to derive that $n\gg H(\gamma)$ when $\gamma$ is  composed of prime elements from $\mathcal{S}$, yielding that $H(\gamma)\ll H(\gamma)^s$.  Since $s<1$, the finiteness conclusion of Theorem \ref{thmfinite} follows.

The paper is organized as follows. Throughout Sections \ref{S2}-\ref{S6}, we let $\Phi_{\beta,D}$  be   the IFS and $K_{\beta,D}$ be the self-similar set defined in \eqref{eqIFS} and \eqref{eqKbetaDdef}, respectively.  In Section \ref{S2}, we prove a counting result on rational vectors in self-similar sets on $\R^d$ and give a version for Gaussian rational numbers in self-similar sets in $\C$. In Section \ref{S3}, we use the counting result Theorem \ref{thmAS} obtained in Section \ref{S2} to derive an upper bound estimate for the period length of an eventually periodic coding (which exists) for every Gaussian rational number in $K_{\beta,D}$. In Section~\ref{S4}, we establish several properties of $\Phi_{\beta,D}$ and $K_{\beta,D}$ that are needed later.  Section~\ref{S5} is devoted to number-theoretical preparations. We study the denominator height in detail so that we can establish a result on the multiplicative order in $\Z[i]$ (see Lemma \ref{lemorderlowerEs}), which is crucial in the proof of Theorem \ref{thmfinite} and may be of independent interest. Finally, we present the proof of Theorem \ref{thmfinite} in the last section, followed by a comment on Gaussian rational numbers in diagonal self-affine sets in $\R^2$.

\section{Counting results on rational vectors/Gaussian rational numbers in self-similar sets}\label{S2}

In this section, we first prove a counting result on 
rational vectors in self-similar sets on $\R^d$. Then we introduce the denominator height for Gaussian integers and use it to formulate a counting result for Gaussian rational numbers in self-similar sets in $\C$. 

For a bounded set $A\subset\R^d$ and $N\in\N$, we denote by $Q_N(A)$ the number of rational vectors in $A$ whose components have the common denominator $N$, and  $Q_N^*(A)$ the number of rational vectors in $A$ with denominators at most $N$. That is,
\begin{equation*}\label{eqQnEQsNe}
Q_N(A)=\#\left\{\frac{\mathbf{p}}{N}\in A: \mathbf{p}\in \Z^d\right\}, \quad Q_N^*(A)=\#\left\{\frac{\mathbf{p}}{q}\in A: \mathbf{p}\in \Z^d, 1\leq q\leq N\right\}.
\end{equation*}
Here and afterwards, $\#F$ stands for the cardinality of a set $F$.

In the following result, 
we give upper bound estimates for $Q_N(E)$ and $Q_N^*(E)$ for a self-similar set $E\subset\R^d$. 

\begin{thm}\label{thmCounting}
Let $E\subset \R^d$ be a self-similar set generated by  an  IFS $\Phi$ with similarity dimension $s$. 
Then there exists $C>0$ such that for any $N\in\N$, 
\begin{equation}\label{eqbddQnE}
Q_N(E)\leq CN^{s}, \quad Q_N^*(E)\leq CN^{\min\{2s,s+1\}}.
\end{equation}
\end{thm}

\begin{rem}\label{remarkA}
{\rm (i)} Clearly, the estimates in \eqref{eqbddQnE} are non-trivial only when $s<d$. 

{\rm (ii)}  Theorem \ref{thmCounting} slightly improves \cite[Theorem 4.1]{Schleischitz21} for self-similar sets. In \cite[Theorem 4.1]{Schleischitz21}, when restricted to self-similar sets, \eqref{eqbddQnE} is proved when $s$ is replaced by $\frac{\log\#\Phi}{-\log\rho_{\max}}$, which is greater than $s$ if $\Phi$ is not homogenous. Here $\rho_{\max}$ denotes the largest similarity ratio of the elements in $\Phi$.

{\rm (iii)} Theorem \ref{thmCounting} also holds when $E$ is an $s$-Ahlfors regular set (not necessarily self-similar); see Remark \ref{remAhlfors} for details.
\end{rem}

As outlined in the introduction, to prove Theorem \ref{thmCounting}, we first give a general counting result for rational vectors in an arbitrary bounded subset of $\R^d$ in terms of its upper box-counting dimension. This is done by adopting an argument from \cite[Theorem 4.1]{Schleischitz21}. Then we show that with the aid of a simple property of self-similar sets, the estimate can be strengthened by slightly modifying the proof, which results in Theorem \ref{thmCounting}.

For a bounded set \(E\subset \mathbb{R}^d\), we denote by \( \mathcal{N}_r(E) \) the smallest number of closed balls of radius \( r \) needed to cover \( E \). The \textit{upper} and \textit{lower box-counting dimensions} of \( E\) are defined by
\[
\overline{\dim}_{\mathrm{B}}E = \varlimsup_{r \to 0^+} \frac{\log \mathcal{N}_r(E)}{-\log r}, \quad \underline{\dim}_{\mathrm{B}}E= \varliminf_{r \to 0^+} \frac{\log \mathcal{N}_r(E)}{-\log r}.
\]
If the upper and lower box-counting dimensions coincide, we call the common value the \textit{box-counting dimension} of $E$.

\begin{prop}\label{thmdimB}
Let $E\subset\R^d$ be a bounded set with $\overline{\dim}_{\rm B}E=s$. Then for any $\epsilon>0$ there exists $C_{\epsilon}>0$ such that for any $N\in\N$, we have 
\[Q_N(E)\leq C_{\epsilon}N^{s+\epsilon}\quad \text{ and } \quad Q_N^*(E)\leq C_{\epsilon}N^{\min\{1+s+\epsilon,\,2(s+\epsilon)\}}.\]
\end{prop}

\begin{proof}
We follow an argument in \cite[Theorem 4.1]{Schleischitz21}.
Fix $\epsilon>0$. Since $\overline{\dim}_{\rm B}E=s$, there exists $C_{\epsilon}>0$ such that for any $N\in\N$,
\begin{equation}\label{eqCover}
\mathcal{N}_{\frac{1}{2N}}(E)\leq C_{\epsilon}N^{s+\epsilon} \quad \text{ and }\quad \mathcal{N}_{\frac{1}{2N^2}}(E)\leq C_{\epsilon}N^{2(s+\epsilon)}.
\end{equation}
Notice that for distinct $\mathbf{p}_1,\mathbf{p}_2\in\Z^d$, and distinct rational vectors $\frac{\mathbf{p}}{q}, \frac{\mathbf{r}}{s}$ with $1\leq q, s\leq N$, we have 
\[\left|\frac{\mathbf{p}_1}{N}-\frac{\mathbf{p}_2}{N}\right|\geq \frac{1}{N},\quad  \left|\frac{\mathbf{p}}{q}-\frac{\mathbf{r}}{s}\right|\geq \frac{1}{qs}\geq \frac{1}{N^2}.\]
This implies that 
\[Q_N(E)\leq \mathcal{N}_{\frac{1}{2N}}(E)\quad \text{ and } \quad Q_N^*(E)\leq \mathcal{N}_{\frac{1}{2N^2}}(E).\]
Then it follows from \eqref{eqCover} that 
\begin{equation}\label{eqQNEstar}
Q_N(E)\leq C_{\epsilon}N^{s+\epsilon}\quad \text{ and } \quad Q_N^*(E)\leq C_{\epsilon}N^{2(s+\epsilon)}.
\end{equation}
To show that $Q_N^*(E)\leq C_{\epsilon}N^{1+s+\epsilon}$, just note that by the first inequality in \eqref{eqQNEstar},
\[Q_N^*(E)\leq \sum_{q=1}^{N}Q_{q}(E)\leq \sum_{q=1}^{N}C_{\epsilon}q^{s+\epsilon}\leq C_{\epsilon}N^{1+s+\epsilon}.\]
This proves the proposition. 
\end{proof}

Now we prove Theorem \ref{thmCounting} following the strategy outlined above.

\begin{proof}[Proof of Theorem \ref{thmCounting}]
Since $E\subset \R^d$ is a self-similar set, $E$ satisfies a  property stating that there exists $C'>0$ such that 
\begin{equation}\label{eqNboud}
\mathcal{N}_{\delta}(E)\leq C'\delta^{-s}, \quad \forall\, 0<\delta<1.
\end{equation}
This is well-known and can be seen from the proof  of \cite[Theorem 9.3]{Falconer14}. 
Then the conclusion of the theorem follows from exactly the same proof as in Proposition~\ref{thmdimB} by replacing \eqref{eqCover} with
\[\mathcal{N}_{\frac{1}{2N}}(E)\leq CN^{s} \quad \text{ and }\quad \mathcal{N}_{\frac{1}{2N^2}}(E)\leq CN^{2s},\]
where $C>0$ is a constant independent of $N$. 
\end{proof}

\begin{rem}\label{remAhlfors}
From the proof of Theorem \ref{thmCounting} we see that the estimates in \eqref{eqbddQnE} actually hold for any compact set $E\subset \R^d$ satisfying \eqref{eqNboud}. This is the case when $E$ is an $s$-Ahlfors regular set.   
Recall that a compact set $E\subset\R^d$ is said to be $s$-Ahlfors regular with $s\in (0, d]$, if there exist a Borel probability measure $\mu$ supported on $E$ and a constant $c>0$ such that 
\begin{equation}\label{eqAhlforsregular}
c^{-1}r^s\leq \mu(B(x,r))\leq cr^s,  \quad \forall x\in E,\,\forall 0<r\leq {\rm diam}(E),
\end{equation}
where $B(x,r)$ is the closed ball centered at $x$ of radius $r$, and ${\rm diam}(E)$ denotes the diameter of $E$. 
It is well-known that an $s$-Ahlfors regular set $E\subset\R^d$ satisfies \eqref{eqNboud};  a short proof is as follows: Fix $0<\delta<1$. Let $M$  be the largest number of disjoint closed balls centered in $E$ of radius $\frac{\delta}{2}$.  Then the maximality of $M$ implies that $\mathcal{N}_{\delta}(E)\leq M$. On the other hand, by \eqref{eqAhlforsregular} and considering the total measure of these balls we see that $c^{-1}(\frac{\delta}{2})^sM\leq 1$, yielding that $M\leq 2^sc\delta^{-s}$. Then \eqref{eqNboud} follows by taking $C'=2^sc$. The above comment on Ahlfors regular sets is inspired by the proof of \cite[Lemma 2.3]{KLW25}.
\end{rem}

For the purpose of this paper, we would like to formulate a version of Theorem~\ref{thmCounting} for Gaussian rational numbers in self-similar sets in $\C$. To this end, we introduce the denominator height for Gaussian integers, which plays a key role in this paper.

Given $\alpha=a+bi\in \Z[i]$, the {\em norm} of $\alpha$ is defined as
\begin{equation}\label{eqnorm}
N(\alpha)=a^2+b^2.
\end{equation}
We define the {\em denominator height}\footnote{This should not be confused with the standard notion of \emph{reduced norm} in algebraic number theory, which is typically defined as $N(\alpha) / \gcd(a, b)^2$. Our function $H(\alpha)$ is tailored to better handle the denominators that appear when studying Gaussian rational numbers, as shown below.} of a non-zero $\alpha$ as 
\begin{equation}\label{defReducedNorm}
H(\alpha)=\frac{N(\alpha)}{\gcd(a,b)}.
\end{equation}
We set $H(0)=0$. The denominator height can be regarded as a notion of height for Gaussian integers, well-suited to the problems studied in this paper. Indeed, for any Gaussian rational number of the form $\frac{\omega}{\alpha}$ with $\alpha\neq0$, we have 
\[\frac{\omega}{\alpha}=\frac{\omega\bar{\alpha}}{N(\alpha)}=\frac{\omega\alpha'}{H(\alpha)},\]
where $\alpha'=\bar{\alpha}/\gcd(a,b)$. Using $H$, we can easily formulate a counting result in $\C$ (Theorem \ref{thmAS}) that is consistent in form with the one in $\R^d$ (Theorem~\ref{thmCounting}). Moreover, $H$ also possesses several other desirable properties (cf. Lemma \ref{LemHsubmul} and Lemmas \ref{LemExip}--\ref{LemIIIpower}), which are important in our subsequent development. 

For a bounded set $A\subset \C$ and $N\in\N$, we define 
\[R_{N}(A)=\#\left\{\frac{\omega}{\gamma} \in \Q(i)\cap A: \omega,\gamma\in\Z[i], H(\gamma)=N\right\},\]
\[R_{N}^*(A)=\#\left\{\frac{\omega}{\gamma} \in \Q(i)\cap A:\omega,\gamma\in\Z[i], 1\leq H(\gamma)\leq N\right\}.\]
We give a version of Theorem~\ref{thmCounting} for Gaussian rational numbers in self-similar sets in $\C$ as follows.

\begin{thm}\label{thmAS}
Let $E\subset \C$ be a self-similar set generated by an IFS $\Phi$ with similarity dimension $s$. Then there exists $C>0$ such that for any $N\in\N$, we have 
\[R_N(E)\leq CN^{s} \quad \text{ and } \quad R_N^*(E)\leq CN^{\min\{2s,s+1\}}.\]
\end{thm}

\begin{proof}
Note that any Gaussian rational number $\frac{\omega}{\gamma}\in \Q(i)$, when written coordinate-wise, is of the form $\frac{\mathbf{p}}{H(\gamma)}\in \Q^2$ for some $\mathbf{p}\in \Z^2$.  Then the theorem follows from  Theorem \ref{thmCounting} since we have $R_N(E)\leq Q_N(E)$, and  $R_{N}^*(E)\leq Q_{N}^*(E)$.
\end{proof}

\section{Period estimate for Gaussian rational numbers in $K_{\beta,D}$}\label{S3}

In this section, we apply the counting result (Theorem \ref{thmAS}) to give an upper bound estimate for the period length of an eventually periodic coding for every Gaussian rational number in $K_{\beta,D}$.

Recall that $K_{\beta,D}$ is the self-similar set generated by  the IFS $\Phi_{\beta,D}$ defined in \eqref{eqIFS}. 
Given $\xi\in K_{\beta,D}$, we call $(j_1j_2\ldots)\in\{1,\ldots,\ell\}^{\N}$ a {\em coding} of $\xi$  if 
\[\xi=\lim_{n\to\infty}\phi_{j_1}\circ\cdots\circ\phi_{j_n}(0)=\sum_{i=1}^{\infty}\frac{t_{j_i}}{\beta^i}.\]
The aim of this section is to prove the following result, which is important in the proof of Theorem \ref{thmfinite}.

\begin{thm}\label{lemperiod}
Let $s$ be the similarity dimension of $\Phi_{\beta,D}$.
Then there exists $C_1>0$ such that every  $\frac{\omega}{\gamma}\in K_{\beta,D}\cap \Q(i)$ has an eventually periodic coding whose period length is at most $C_1H(\gamma)^{\min\{s,2\}}$. 
\end{thm}

To prove Theorem~\ref{lemperiod}, we employ an argument from the proof of \cite[Theorem 4.3]{Schleischitz21}.
We will see from the proof that the conclusion actually holds for self-similar sets generated by more  general IFSs,  which include for instance $\{\beta_j^{-1}z+\xi_j\}_{j=1}^{\ell}$, where $\beta_j\in\Z[i], \xi_j\in\Q(i)$, $1\leq j\leq \ell$. In that case, as noted in Remark~\ref{remarkA}(ii), Theorem~\ref{lemperiod} slightly improves the exponent in the estimate in \cite[Theorem 4.3]{Schleischitz21}.

We need the following lemma, which shows that the denominator height $H$ defined in \eqref{defReducedNorm} is sub-multiplicative on $\Z[i]$. More properties of $H$ will be established in Section \ref{S5}.

\begin{lem}\label{LemHsubmul}
Let $\alpha,\gamma\in\Z[i]$. Then the following hold.
\begin{itemize}
  \item[(i)] If $\alpha\in \Z\cup(i\Z)$, then $H(\alpha\gamma)=|\alpha|H(\gamma)$.
  \item[(ii)] $H$ is sub-multiplicative: $H(\alpha\gamma)\leq H(\alpha)H(\gamma)$.
\end{itemize}
\end{lem}

\begin{proof}
Part (i) can be easily verified by the definition of $H$. Below we prove (ii). 
Let $\alpha=a+bi, \gamma=c+di$, where $a,b,c,d\in\Z$. Then $$\alpha\gamma=ac-bd+(ad+bc)i.$$ 
Thus we have
\[H(\alpha\gamma)=\frac{N(\alpha\gamma)}{\gcd(ac-bd, ad+bc)}, \quad H(\alpha)H(\gamma)=\frac{N(\alpha)N(\gamma)}{\gcd(a,b)\gcd(c,d)}.\]
Since $N(\alpha\gamma)=N(\alpha)N(\gamma)$, it suffices to show that
\begin{equation}\label{eqgcdac}
\gcd(ac-bd, ad+bc)\geq \gcd(a,b)\gcd(c,d).
\end{equation} 
To see this, write $s=\gcd(a,b), t=\gcd(c,d)$ and assume $a=sa_1, b=sb_1, c=tc_1,d=td_1$, where $a_1,b_1,c_1,d_1\in\Z$. Then 
\[ac-bd=st(a_1c_1-b_1d_1), \quad ad+bc=st(a_1d_1+b_1c_1),\]
which yields that $st\mid\gcd(ac-bd, ad+bc)$. Hence \eqref{eqgcdac} holds and we finish the proof.  
\end{proof}

Now we give the proof of Theorem \ref{lemperiod}.

\begin{proof}[Proof of Theorem \ref{lemperiod}]
Let \(\xi=\frac{\omega}{\gamma}\in K_{\beta,D}\cap\mathbb{Q}(i) \), where $\omega,\gamma\in \Z[i], \gamma\neq0$. Let \((j_1j_2\ldots)\in \{1,\ldots,\ell\}^{\N}\) be a coding  of \(\xi\), i.e.,
\[\xi=\lim_{n\to\infty}\phi_{j_1}\circ\cdots\circ\phi_{j_n}(0)=\sum_{i=1}^{\infty}\frac{t_{j_i}}{\beta^i}.\]
We consider the sequence \[\xi_0 = \xi,\quad \xi_1=\phi_{j_1}^{-1}(\xi),\quad \xi_2 = \phi_{j_2}^{-1} \circ \phi_{j_1}^{-1}(\xi),  \ldots.\] 
Clearly, \( \xi_k \in K_{\beta,D}\cap\mathbb{Q}(i)\)  for \(k\geq 0 \). Let $s_j\in \Z[i]$ be the denominator of $t_j$ in the IFS $\Phi_{\beta,D}$ (cf. \eqref{eqIFS}), $1\leq j\leq \ell$, and let $\Gamma=s_1\cdots s_{\ell}$. Then from the expressions of the mappings $\phi_1,\ldots, \phi_{\ell}$ in $\Phi_{\beta,D}$, we see that the denominators of $\xi_k$ ($k\geq 0$) are all divisors of $\gamma \Gamma$.  Notice that such Gaussian rational numbers are $\frac{1}{H(\gamma\Gamma)}$-separated, and by Lemma \ref{LemHsubmul} we have $H(\gamma\Gamma)\leq H(\gamma)H(\Gamma)$. Since \( K_{\beta,D} \) is compact, it follows that there are at most 
\[ (H(\gamma\Gamma){\rm diam}(K_{\beta,D})+1)^2 \ll H(\gamma)^2 \]
many such Gaussian rational numbers in $K_{\beta,D}$. 
On the other hand, by  Theorem \ref{thmAS} there are at most \( \ll H(\gamma)^s \) such Gaussian rational numbers in $K_{\beta,D}$. Hence the cardinality of the set $\{\xi_k:k\geq 0\}$ is \( \ll H(\gamma)^{\min\{s,2\}} \). This implies that \(\xi_m = \xi_n \) for some \( m<n\ll H(\gamma)^{\min\{s,2\}}\), which means that 
\[ \xi_m = \phi_{j_{m+1}} \circ \phi_{j_{m+2}} \circ \cdots \circ \phi_{j_{n}}(\xi_m).\]
From this we see that $\xi_{m}=\lim_{k\to\infty}\psi^k(0)$, where $\psi=\phi_{j_{m+1}} \circ \phi_{j_{m+2}} \circ \cdots \circ \phi_{j_{n}}$ and $\psi^k$ is the $k$-fold composition of $\psi$. It follows that $\xi_m$ has a coding $(j_{m+1}j_{m+2}\ldots j_n)^{\infty}$. 
Thus $\xi$ has a coding \( (j_1\ldots j_m(j_{m+1}j_{m+2}\ldots j_n)^{\infty})\) with $n\ll H(\gamma)^{\min\{s,2\}}$.
\end{proof}

\section{More preliminaries on $\Phi_{\beta,D}$ and $K_{\beta,D}$}\label{S4}

In this section, we establish several properties of $\Phi_{\beta,D}$ and $K_{\beta,D}$. The main objects are Lemmas~\ref{LemWSP} and \ref{LemComposition}, which will enable us to pass from the assumption that $\dim_{\rm H}K_{\beta,D}<1$ to that certain composition of $\Phi_{\beta,D}$ has similarity dimension less than $1$. Our strategy in this section is similar to that in \cite{KLW25}. A difference is that we choose to use Lemma \ref{LemComposition} rather than referring to results in \cite{FHOR15O}, which makes this paper more self-contained.

We first introduce some notation.
Given an IFS $\Phi=\{\phi_i\}_{i=1}^{\ell}$ on $\R^d$ which generates a self-similar set $E$, let $\Sigma=\{1,\ldots, \ell\}$ be the alphabet associated to $\Phi$. For each $n\in\N$, let $$\Sigma_n=\{i_1i_2\ldots i_n: i_k\in\Sigma, 1\leq k\leq n\}.$$ Moreover, set $\Sigma_0=\{\varepsilon\}$, where $\varepsilon$ denotes the empty word. Let $\Sigma_*=\bigcup_{n=0}^{\infty}\Sigma_n$ and $\Sigma^{\N}$ be the collection of infinite words over $\Sigma$. For $I=i_1\ldots i_n\in \Sigma_n$, write $\phi_{I}=\phi_{i_1}\circ\cdots\circ\phi_{i_n}$ and let $\rho_I$ denote the similarity ratio of $\phi_I$. In particular, $\rho_{\varepsilon}=1$ and $\phi_{\varepsilon}:=id$, the identity map on $\R^d$. We use $\Phi^n$ to denote the $n$-fold composition of $\Phi$. That is, $\Phi^n=\{\phi_I: I\in\Sigma_n\}$. More generally, given $b\in (0,1)$, let 
\begin{equation}\label{eqPhib}
\Phi_{b}=\{\phi_{I}: I\in \Sigma_*, b\rho_{\min}<\rho_{I}\leq b\},
\end{equation}
where $\rho_{\min}=\min\{\rho_1,\ldots, \rho_{\ell}\}$. It is well-known that $\Phi_b$ and $\Phi^n$ are both generating IFSs for $E$, i.e., $E=\bigcup_{f\in \Phi_b}f(E)=\bigcup_{g\in\Phi^n}g(E)$; see e.g. \cite{Zerner96}.
 Set 
 \begin{equation}\label{eqFb}
\mathcal{F}=\{\phi_{I}^{-1}\circ\phi_{J}: I,J\in\Sigma_*\}, \quad \mathcal{F}_b=\{f^{-1}\circ g: f,g\in \Phi_b\}.
 \end{equation}
We say that $\Phi$ satisfies the {\em weak separation property} (WSP) if the identity map is not an accumulation point of $\mathcal{F}$ (cf. \cite{LN99M} and \cite{Zerner96}).
 
We will show that the IFS $\Phi_{\beta,D}$ defined in \eqref{eqIFS} satisfies the WSP whenever $K_{\beta,D}$ is not contained in  a line in $\C$. As a consequence, an approximation property holds (see Lemma \ref{LemComposition}), which will enable us to use the assumption that $\dim_{\rm H}K_{\beta,D}<1$ in the proof of Theorem \ref{thmfinite}. To this end, we make use of the following characterization of the WSP due to Zerner \cite{Zerner96}. 

\begin{prop}\cite[Theorem 1]{Zerner96}\label{propWSC}
Let \(E\subset\R^d\) be a self-similar set generated by an IFS \(\Phi\). Suppose $E$ is not contained in a hyperplane in $\R^d$. Then  \(\Phi\) satisfies the WSP if and only if there exist \( x_0 \in \mathbb{R}^d \) and \( \epsilon > 0 \) such that  for any \(b\in (0,1)\),
\[
\forall\, h \in \mathcal{F}_{b}, \quad h(x_0) \neq x_0 \implies |h(x_0) - x_0| > \epsilon.
\]
\end{prop}

By applying Proposition \ref{propWSC}, we have the following.

\begin{lem}\label{LemWSP}
If $K_{\beta,D}$ is not contained in a line in $\C$, then $\Phi_{\beta,D}$ satisfies the WSP. 
\end{lem}

\begin{proof}
Fix $b\in (0,1)$. Since $\Phi:=\Phi_{\beta,D}$ is homogeneous, we see from the definitions of $\Phi_b$ and $\mathcal{F}_b$ (cf. \eqref{eqPhib}-\eqref{eqFb}) that there exists $n\in\N$ such that  $\mathcal{F}_{b}=\{\phi_{I}^{-1}\circ \phi_{J}: I,J\in \Sigma_n\}$. Let $h=\phi_{I}^{-1}\circ \phi_{J}$ for some $I=i_1\ldots i_{n}, J=j_1\ldots j_n\in \Sigma_n$. Then 
\[h(x)=x+\sum_{k=1}^{n}t_{j_k}\beta^{n-k}-\sum_{k=1}^{n}t_{i_k}\beta^{n-k}.\]
Since $D=\{t_1,\ldots, t_{\ell}\}\subset \Q(i)$ is a finite set, we can find a non-zero $\Theta\in \Z[i]$ so that $\Theta t_j\in \Z[i]$ for all $1\leq j\leq \ell$. Therefore, if $h(0)\neq 0$, then we have
\[|\Theta h(0)|=\left|\sum_{k=1}^{n}\Theta t_{j_k}\beta^{n-k}-\sum_{k=1}^{n}\Theta t_{i_k}\beta^{n-k}\right|\geq 1,\]
since $\Theta h(0)\in \Z[i]\setminus\{0\}$. Let $x_0=0$ and $\epsilon=1/(2|\Theta|)$. Since $b\in (0,1)$ is arbitrary,  it follows from Proposition \ref{propWSC} that $\Phi_{\beta,D}$ satisfies the WSP. 
\end{proof}

In our proof of Theorem \ref{thmfinite}, we  will make use of the following property for IFSs satisfying the WSP.

\begin{lem}\label{LemComposition}
Let $E\subset\R^d$ be a self-similar set  generated by an IFS $\Phi$ satisfying the WSP. Suppose $E$ is not contained in a hyperplane in $\R^d$. Then for any $t>\dim_{\rm H}E$ there exists a generating IFS $\Psi$ of $E$ having similarity dimension less than $t$. If $\Phi$ is homogeneous, then $\Psi$ can be chosen to be $\Phi^n$ for some $n\in\N$.
\end{lem}

\begin{proof}
This is a simple consequence of the results in \cite{Zerner96}. Indeed, according to Theorem 2 and Lemma 2 in \cite{Zerner96}, we have $\dim_{\rm H}E=\lim_{b\to0}s_b$, where $s_b$ is the similarity dimension of the IFS $\Phi_{b}$ defined in \eqref{eqPhib}. Then for any $t>\dim_{\rm H}E$, the first part of the lemma follows by taking $\Psi=\Phi_b$ with $s_b<t$. To see the second part of the lemma, simply notice that when $\Psi$ is homogeneous, $\Phi_b=\Phi^n$ for some $n\in\N$.
\end{proof}

Finally, we give a simple fact on $K_{\beta,D}$.

\begin{lem}\label{lemnotline}
If $\beta\in\Z[i]\setminus \Z$ and $K_{\beta,D}$ is not a singleton, then $K_{\beta,D}$ is not contained in a line in $\C$. 
\end{lem}

\begin{proof}
Since  $K_{\beta,D}$ is not a singleton, there exist two  similitudes in the IFS $\Phi_{\beta,D}$ defined in \eqref{eqIFS} having distinct fixed points. Without loss of generality, we assume  $\phi_1$ and $\phi_2$ have distinct fixed points. Then $t_1\neq t_2$. Let $u$ be the fixed point of $\phi_1$. Set $v=\phi_2(u)$ and $w=\phi_1(v)$. Clearly, $u,v, w\in K_{\beta,D}$. Moreover,  notice that 
\[v-u=\phi_2(u)-\phi_1(u)=\frac{t_2-t_1}{\beta}, \quad w-u=\phi_1\circ\phi_2(u)-\phi_1\circ\phi_1(u)=\frac{t_2-t_1}{\beta^2}.\]
Since $\beta\in\Z[i]\setminus \Z$, it follows that $v-u$ and $w-u$ are not colinear. Hence $K_{\beta,D}$ is not contained in a line in $\C$.
\end{proof}

\section{Number-theoretical preliminaries}\label{S5}

This section is devoted to number-theoretical preparations that are needed to prove Theorem \ref{thmfinite}. The main result is Lemma \ref{lemorderlowerEs}, which plays a key role in the proof of Theorem \ref{thmfinite} and may be of some independent interest. 

We begin with some notation. 
Let $\alpha,\gamma\in \Z[i] $.
 We say that $\alpha, \gamma$ are {\em coprime} and write $\gcd(\alpha,\gamma)=1$, if the only  common divisors of $\alpha$ and $\gamma$ are the units in $\Z[i]$. In this case,  we define the {\em multiplicative order} of $\alpha$ modulo $\gamma$, denoted by ${\rm ord}(\alpha;\gamma)$, to be the smallest positive integer $k$ such that $\alpha^k\equiv 1\pmod \gamma$. This order is well-defined because for $\gcd(\alpha,\gamma)=1$, the residue class $\alpha\pmod \gamma$ is an element of the finite multiplicative group $(\Z[i]/(\gamma))^{\times}$, which ensures that such a $k$ exists. We adopt the convention that \({\rm ord}(\alpha; 1) = 1\). If $\gamma$ is a prime element and $\alpha$ is non-zero, we define the {\em $\gamma$-adic valuation} of $\alpha$, denoted by $\nu_{\gamma}(\alpha)$, to be the largest non-negative integer $m$ so that $\gamma^m\mid \alpha$. 
 
Lemma \ref{lemdivord} is a simple property of the multiplicative order in $\Z[i]$, whose proof is included for completeness. 
 
\begin{lem}\label{lemdivord}
Let $\alpha,\beta,\gamma\in\Z[i]$ and assume that $\beta\mid\gamma$ and  ${\rm gcd}(\alpha,\gamma)=1$. Then ${\rm ord}(\alpha;\beta)|{\rm ord}(\alpha;\gamma)$. 
\end{lem}

\begin{proof}
By definition,  $\gamma\mid (\alpha^{{\rm ord}(\alpha;\gamma)}-1)$. Since $\beta\mid\gamma$, we thus have $\beta\mid (\alpha^{{\rm ord}(\alpha;\gamma)}-1)$ and so $\alpha^{{\rm ord}(\alpha;\gamma)}\equiv 1\pmod \beta$. It follows that ${\rm ord}(\alpha;\beta)|{\rm ord}(\alpha;\gamma)$. 
\end{proof}

Recall that $N(\cdot)$ is the norm and $H(\cdot)$ is the denominator height on $\Z[i]$; see \eqref{eqnorm}-\eqref{defReducedNorm}. Since $\Z[i]$ is a UFD, every element $z\in\Z[i]\setminus\{0\}$ can be written uniquely (up to order and unit multiples) as 
\[z=u\beta_1^{n_1}\beta_2^{n_2}\cdots\beta_k^{n_k},\]
where $u$ is a unit, and $\beta_1,\ldots,\beta_k$ are non-associate prime elements.
 
The development in the rest of this section relies heavily on the following classification of prime elements in $\Z[i]$.

\begin{lem}\cite[Proposition 3.68]{Rotman02A}\label{LemClassify}
A Gaussian integer $\gamma\in\Z[i]$ is a prime element if and only if

(Type I) \( \gamma \) is an associate of a prime number in \( \mathbb{N} \) of the form \(4m + 3 \); or

(Type II) \( \gamma \) is an associate of \( 1 + i \); or

(Type III) \( N(\gamma)\) is a prime number in \( \mathbb{N} \) of the form \( 4m + 1 \).
\end{lem}

The following is a simple property of the norm $N$ in $\Z[i]$. 

\begin{lem}\label{lemgcd}
Let $\alpha\in\Z[i]$ and $\gamma\in\Z\cup (i\Z)$. Then $\gcd(\alpha,\gamma)=1$ in $\Z[i]$ if and only if $\gcd(N(\alpha),N(\gamma))=1$ in $\Z$. 
\end{lem}

\begin{proof}
The ``if" part actually holds for all $\alpha,\gamma\in\Z[i]$. To see this, let $\alpha,\gamma\in \Z[i]$ with $\gcd(N(\alpha),N(\gamma))=1$ in $\Z$. If $\alpha,\gamma$ are not coprime, then there exists a prime element $\eta\in \Z[i]$ such that $\eta\mid \alpha$ and $\eta \mid \gamma$. It follows that $N(\eta)\mid N(\alpha)$ and $N(\eta)\mid N(\gamma)$. Since $N(\eta)>1$ as $\eta$ is not a unit, this contradicts the assumption that  $\gcd(N(\alpha),N(\gamma))=1$ in $\Z$. 

To prove the ``only if" part, we focus on the case that $\gamma\in \Z$; the other case that $\gamma\in i\Z$ can be proved similarly. Let $\alpha\in\Z[i], \gamma\in \Z$ and assume that $\gcd(\alpha,\gamma)=1$ in $\Z[i]$. If $\gcd(N(\alpha),N(\gamma))>1$ in $\Z$, then we can find a prime number $p\in \N$ dividing both $N(\alpha)$ and $N(\gamma)$. Since $\gamma\in\Z$, we have $N(\gamma)=\gamma^2$ and so $p\mid \gamma$. On the other hand, write $\alpha=u\alpha_1^{k_1}\cdots\alpha_s^{k_s}$, where $u$ is a unit in $\Z[i]$, $k_1,\ldots,k_s\in\N$, and $\alpha_1,\ldots,\alpha_s\in \Z[i]$ are non-associate prime elements. Then since $N(\alpha)=N(\alpha_1)^{k_1}\cdots N(\alpha_s)^{k_s}$, $p\mid N(\alpha)$ and $p\in\N$ is a prime number, we can assume without loss of generality that $p\mid N(\alpha_1)$. Below we discuss the type of $\alpha_1$ to derive a contradiction. 

If $\alpha_1$ is a prime element of Type I, then $\alpha_1\in \Z\cup (i\Z)$ with $|\alpha_1|\in \N$ being a prime number. Since $p\mid N(\alpha_1)=|\alpha_1|^2$, we see that $p=|\alpha_1|$. Therefore, $\alpha_1$ is a common non-unit divisor of $\alpha$ and $\gamma$, a contradiction. If $\alpha_1$ is of Type II or Type III, then $N(\alpha_1)\in\N$ is prime and so $p=N(\alpha_1)$. Since $\alpha_1\mid N(\alpha_1)$ in $\Z[i]$, we again see that  $\alpha_1$ divides both $\alpha$ and $\gamma$, which again contradicts that $\gcd(\alpha,\gamma)=1$. This completes the proof of the lemma. 
\end{proof}

The following lemma is about properties of the denominator height of prime elements, which are needed in the proof of Lemma \ref{lemcompord}. 

\begin{lem}\label{LemExip}
Let  $\gamma\in \Z[i]$ be a prime element. Then  $H(\gamma)$ is a prime number in $\N$ and $\gamma\mid H(\gamma)$ in $\Z[i]$. Moreover,  $\nu_{\gamma}(H(\gamma))=1$ if $\gamma$ is of Type I or Type III and $\nu_{\gamma}(H(\gamma))=2$ if $\gamma$ is of Type II. 
\end{lem}

\begin{proof}
From the classification of the prime elements in $\Z[i]$ (see Lemma \ref{LemClassify}) and the definition of the denominator height $H$ (see \eqref{defReducedNorm}), we easily see that 
\[H(\gamma)=\begin{cases}
                      |\gamma|, & \text{ if }\gamma \text{ is of Type I},\\
                      2,       & \text{ if }\gamma \text{ is of Type II},\\
                      N(\gamma),& \text{ if }\gamma \text{ is of Type III}. 
                    \end{cases}\]
Then it is straightforward to see that in each case,  $H(\gamma)$ is a prime number in $\N$ and $\gamma\mid H(\gamma)$ in $\Z[i]$.  For the second statement, the conclusion clearly holds if $\gamma$ is of Type I. If $\gamma$ is of Type II, then it is easy to check that $H(\gamma)=2=\gamma^2u$ for some unit $u\in \Z[i]$.  For the case when $\gamma$ is of Type III,  we argue by contradiction. Suppose $H(\gamma)=\gamma^2\omega$ for some $\omega\in \Z[i]$. Write $\gamma=a+bi$, where $a,b\in\Z$. Then since $N(\gamma)=a^2+b^2$ is a prime number of the form $4m+1$, $a,b$ are both non-zero and $a\neq \pm b$. Moreover, since $H(\gamma)=N(\gamma)=\gamma\bar{\gamma}$, we have $\bar{\gamma}=\gamma\omega$ and so
\[\omega=\frac{\bar{\gamma}}{\gamma}=\frac{a^2-b^2}{a^2+b^2}-\frac{2ab}{a^2+b^2}i,\]
which is not in $\Z[i]$, a contradiction. Hence $\nu_{\gamma}(H(\gamma))=1$ when $\gamma$ is a prime element of Type III. This finishes the proof.
\end{proof}

 For $x\in\R$, let $\lfloor x \rfloor$ denote the largest integer less than or equal to $x$. In the following two lemmas, we establish properties of the denominator height for powers of prime elements. 

\begin{lem}\label{LemHalphaprime}
Let $\gamma\in \Z[i]$ be a prime element. Then for any $n\in\N$,
\[ H(\gamma^n)=
\begin{cases}
  |\gamma|^n, & \text{ if }\gamma \text{ is  of Type I},\\
  2^{n-\lfloor\frac{n}{2}\rfloor}, & \text{ if }\gamma \text{ is of Type II},\\
  N(\gamma)^n,& \text{ if }\gamma \text{ is of Type III}.
\end{cases}
\]
In particular, $H(\gamma^n)\leq 2\cdot 2^{\frac{n}{2}}$ for $n\in\N$ if $\gamma$ is of Type II.
\end{lem}

\begin{proof}
If $\gamma$ is a Type I, then $\gamma=up$ for some unit $u\in\Z[i]$ and a prime number $p\in\N$ of the form $4m+3$. Hence for $n\in\N$, we have
\[H(\gamma^n)=H(u^np^n)=p^n=|\gamma|^n,\]
where the second equality holds by Lemma \ref{LemHsubmul}(i).

For the case that $\gamma$ is of Type II, since $H$ takes the same value at a Gaussian integer and its associates (cf. Lemma~\ref{LemHsubmul}(i)), it suffices to prove the lemma for $\gamma=1+i$. A simple calculation gives that $\gamma^n=a_n+b_ni$, where
\[a_n=2^{n/2}\cos(n\pi/4),\quad b_n=2^{n/2}\sin(n\pi/4).\]
From this we easily see that $\gcd(a_n,b_n)=2^{\lfloor\frac{n}{2}\rfloor}$. Since $N(\gamma^n)=N(\gamma)^n=2^n$, the conclusion follows.

Finally, we consider the case that $\gamma$ is of Type III. Then $N(\gamma):=p\in \N$ is a prime number of the form $4m+1$. For $n\in\N$, write $\gamma^n=A_n+B_ni$, where $A_n,B_n\in \Z$. Since 
\[H(\gamma^n)=\frac{N(\gamma^n)}{\gcd(A_n,B_n)}=\frac{N(\gamma)^n}{\gcd(A_n,B_n)},\]
it suffices to prove that $\gcd(A_n,B_n)=1$. We argue by contradiction. Suppose that $q\in\N$ is a prime number dividing both $A_n$ and $B_n$. Then certainly $q\mid \gamma^n$ in $\Z[i]$. We consider three  scenarios separately as follows. 

{\em Case 1}. $q\equiv 3\pmod 4$. In this case, $q$ is a prime element of Type I. Hence $q\mid \gamma^n$ implies that $q\mid \gamma$, and so $q^2\mid N(\gamma)=p$. A contradiction arises since $p\in \N$ is a prime number. 

{\em Case 2}. $q=2$. Since $2=(1+i)(1-i)$ and $1+i$ is a prime element in $\Z[i]$, we see from $2\mid \gamma^n$ that $1+i\mid \gamma$. This yields that $N(1+i)\mid N(\gamma)$, i.e., $2\mid p$, which is again a contradiction since $p$ is odd.

{\em Case 3}. $q\equiv 1\pmod 4$. By Fermat’s Two-Squares Theorem (cf. \cite[Theorem 3.66]{Rotman02A}), there exist $s,t\in\N$  such that $q=s^2+t^2$. Then by Lemma \ref{LemClassify},  $s+ti$ is a prime element of Type III.  Since $q=(s+ti)(s-ti)$ and $q\mid \gamma^n$, we see that $s+ti$ divides $\gamma $. This again leads to a contradiction since $\gamma$ is a prime element.  
\end{proof}

The following lemma deals with the denominator height of products of powers of conjugate Type III prime elements.

\begin{lem}\label{LemIIIpower}
Let $\alpha\in \Z[i]$ be a prime element of Type III and $\gamma$ be an associate of $\bar{\alpha}$. Then for any $m,n\in\N$, 
\[H(\alpha^n\gamma^m)=H(\alpha)^{\max\{m,n\}}.\]
\end{lem}

\begin{proof}
Without loss of generality, assume $\gamma=\bar{\alpha}$ and $n\geq m$; the other cases can be proved similarly. Then we have 
\[\alpha^n\gamma^m=\alpha^{n-m}(\alpha\bar{\alpha})^{m}=N(\alpha)^{m}\alpha^{n-m}.\]
Therefore, by Lemma~\ref{LemHsubmul}(i) and Lemma \ref{LemHalphaprime},
\[H(\alpha^n\gamma^m)=N(\alpha)^{m}H(\alpha^{n-m})=N(\alpha)^{m}N(\alpha)^{n-m}=N(\alpha)^n=H(\alpha)^n.\]
\end{proof}

In the following, we prove an order-lifting lemma for the multiplicative order in $\Z[i]$, which can be viewed as  a complex analogue of \cite[Lemma 3]{Bloshchitsyn15}.

\begin{lem}\label{lemcompord}
Let $\alpha,\gamma\in\Z[i]$ with $\gamma$ a prime element and $\gamma\nmid \alpha$. If  $\gamma$ is of Type I or Type III, set $p=H(\gamma)$, $d={\rm ord}(\alpha;\gamma)$ and $m=\nu_{\gamma}(\alpha^d-1)$. Then we have 
\begin{equation}\label{eqordalbeta2}
{\rm ord}(\alpha;\gamma) ={\rm ord}(\alpha;\gamma^2) = \cdots ={\rm ord}(\alpha;\gamma^m)= d,
\end{equation}
and
\begin{equation}\label{eqordmpj}
 {\rm ord}(\alpha;\gamma^{m+j})=p^j d, \quad \nu_{\gamma}(\alpha^{p^jd}-1)=m+j,\quad \forall j\in\N.
\end{equation}
If $\gamma$ is of Type II,  let $u\in\Z[i]$ be the unit such that $2=u\gamma^2$. Set $m=\nu_{\gamma}(\alpha-1)$ and $t_j=\nu_{\gamma}(\alpha^{2^j}-1)$ for $j\in\N$. Then the following hold.

{\rm (i)} If $m=1$, set $\tau_1=\nu_{\gamma}(\eta^2+u\gamma\eta+u)$, where $\eta=(\alpha-1)\gamma^{-1}$. Then
\begin{equation}\label{equtau1}
t_1=2,\quad t_j=\tau_1+2j \quad (\forall j\geq 2), \quad {\rm ord}(\alpha;\gamma^{t_j})=2^j\quad (\forall j\in\N).
\end{equation}

{\rm (ii)} If $m=2$, set $\tau_2=\nu_{\gamma}(\eta+u)$, where $\eta=(\alpha-1)\gamma^{-2}$. Then 
\begin{equation*}
t_j=\tau_2+2j+2, \quad {\rm ord}(\alpha; \gamma^{t_j})=2^j,\quad \forall j\in\N.
\end{equation*}

{\rm (iii)} If $m\geq 3$, then $t_j=m+2j$ and  ${\rm ord}(\alpha;\gamma^{t_j})=2^j$ for $j\in \N$.
\end{lem}

\begin{proof}
First consider the case that $\gamma$ is of either Type I or Type III. 
Since ${\rm ord}(\alpha;\gamma)=d$ and $\nu_{\gamma}(\alpha^d-1)=m$, we have for $1\leq k\leq m$, $\alpha^d\equiv1\pmod{\gamma^k}$ and so ${\rm ord}(\alpha; \gamma^k)\mid d$. On the other hand, by Lemma \ref{lemdivord} we have $d\mid{\rm ord}(\alpha; \gamma^k)$. Hence ${\rm ord}(\alpha; \gamma^k)=d$ for $1\leq k\leq m$. This proves \eqref{eqordalbeta2}. Below we prove \eqref{eqordmpj} by induction on $j$.

Since $\nu_{\gamma}(\alpha^d-1)=m$, we can write $\alpha^d=1+\omega\gamma^m$, where $\omega\in\Z[i]$, $\gamma\nmid \omega$. Then by the binomial formula, 
\begin{equation}\label{eqBinomial}
\alpha^{pd}=(1+\omega\gamma^m)^{p}=1+p\omega\gamma^m+\binom{p}{2}(\omega\gamma^m)^2+\cdots+\binom{p}{p}(\omega\gamma^m)^p.
\end{equation}
By Lemma \ref{LemExip}, $\gamma\mid p$ and $\nu_{\gamma}(p)=1$. It follows that $\alpha^{pd}\equiv1\pmod{\gamma^{m+1}}$ and $\nu_{\gamma}(\alpha^{dp}-1)=m+1$. Hence ${\rm ord}(\alpha;\gamma^{m+1})\mid pd$. On the other hand, by Lemma \ref{lemdivord}, $d\mid {\rm ord}(\alpha;\gamma^{m+1})$. Moreover, $\nu_{\gamma}(\alpha^d-1)=m$ implies that $ {\rm ord}(\alpha;\gamma^{m+1})\neq d$. Hence $ {\rm ord}(\alpha;\gamma^{m+1})=pd$.  It then follows from \eqref{eqBinomial} that $\nu_{\gamma}(\alpha^{pd}-1)=m+1$. This proves \eqref{eqordmpj} for $j=1$. Suppose that \eqref{eqordmpj} holds for some $j\geq 1$. We are going to prove that it also holds for $j+1$. To this end, note that  by the induction hypothesis, we can write $\alpha^{p^jd}=1+\xi\gamma^{m+j}$, where $\xi\in\Z[i]$ and $\gamma\nmid \xi$. Then again by the binomial formula, we have
\[\alpha^{p^{j+1}d}=(1+\xi\gamma^{m+j})^p=1+p\xi\gamma^{m+j}+\binom{p}{2}(\xi\gamma^{m+j})^2+\cdots+\binom{p}{p}(\xi\gamma^{m+j})^p.\]
Again, since $\gamma\mid p$ and $\nu_{\gamma}(p)=1$, we see that $\nu_{\gamma}(\alpha^{p^{j+1}d}-1)=m+j+1$. Hence ${\rm ord}(\alpha; \gamma^{m+j+1})\mid p^{j+1}d$. It then follows from Lemma \ref{lemdivord} and the  induction hypothesis  that ${\rm ord}(\alpha; \gamma^{m+j+1})=p^{j+1}d$. Therefore, \eqref{eqordmpj} holds for $j+1$. This proves \eqref{eqordmpj}.

Next, we consider the case that $\gamma$ is of Type II. Since $\Z[i]$ is a Euclidean domain with the norm $N$, $\gamma\nmid\alpha$ and $N(\gamma)=2$, we easily see from the division algorithm that $\alpha\equiv1\pmod\gamma$. Hence $m=\nu_{\gamma}(\alpha-1)$ is well-defined.  Below we only present the proof for the case that $m=1$; the other two cases can be verified similarly. 

Since $m=1$, we have $\alpha=1+\gamma\eta$ with $\gamma\nmid \eta$. Then 
\[\alpha^2=(1+\gamma\eta)^2=1+2\gamma\eta+\gamma^2\eta^2=1+\gamma^2\eta(u\gamma+\eta).\]
Since $\gamma\nmid\eta$ and $\gamma\nmid (u\gamma+\eta)$, it follows that $\nu_{\gamma}(\alpha^2-1)=2$ and ${\rm ord}(\alpha;\gamma^2)\mid 2$. Since $m=1$, we have ${\rm ord}(\alpha;\gamma^2)=2$. For $\alpha^{2^2}$ we have 
\begin{align*}
\alpha^{2^2}&=(1+\gamma^2\eta(u\gamma+\eta))^2\\
&=1+2\gamma^2\eta(u\gamma+\eta)+\gamma^4\eta^2(u\gamma+\eta)^2. \\
&=1+\gamma^4\eta(u\gamma+\eta)(\eta^2+u\gamma\eta+u).
\end{align*}
Since $\tau_1=\nu_{\gamma}(\eta^2+u\gamma\eta+u)$, it follows that $\nu_{\gamma}(\alpha^{2^2}-1)=\tau_1+4$. This proves that the statement for $t_j$ in \eqref{equtau1} holds for $j=2$. Then by induction on $j$, using the relation $2=u\gamma^2$ and arguing as in the Type I/III case, we obtain the formulas for $t_j$ and ${\rm ord}(\alpha;\gamma^{t_j})$ for all $j\geq 1$ in \eqref{equtau1}. This finishes the proof of the lemma. 
\end{proof}

The following  is a consequence of Lemma \ref{lemcompord}.

\begin{lem}\label{Lemordbase2}
Let \( \alpha \in \mathbb{Z}[i]\) and let \(\gamma\in\Z[i]\) be a prime element of Type II with \(\gamma\nmid \alpha\). Then for any $n\in\N$, ${\rm ord}(\alpha;\gamma^n)$ is an integer power of $2$ and we have
\[{\rm ord}(\alpha; \gamma^n) \geq C_22^{\frac{n}{2}},
\]
where $C_2>0$ is a constant independent of $n$. 
\end{lem}

\begin{proof}
We have seen in the proof of Lemma \ref{lemcompord} that $\alpha\equiv1\pmod\gamma$. 
Let $m=\nu_{\gamma}(\alpha-1)$. For simplicity, we only consider the case that $m\geq 3$; the other two cases $m=1$ and $m=2$ can be proved similarly with the aid of Lemma \ref{lemcompord} (with possibly different constants $C_2$). For  $n\geq m$, let $j\geq 0$ be such that 
\begin{equation*}\label{eqmp2jj1}
m+2j\leq n<m+2(j+1).
\end{equation*}
Then by Lemma \ref{lemdivord} and Lemma \ref{lemcompord}(iii), ${\rm ord}(\alpha;\gamma^n)$ is an integer power of $2$ and 
\[2^j={\rm ord}(\alpha;\gamma^{m+2j})\leq {\rm ord}(\alpha;\gamma^{n})\leq {\rm ord}(\alpha;\gamma^{m+2(j+1)})= 2^{j+1}.\]
Therefore,
\[{\rm ord}(\alpha;\gamma^{n})\geq 2^j\geq 2^{\frac{n-m}{2}-1}.\]
This inequality clearly also holds when $n<m$. Hence the lemma follows by taking $C_2=2^{-(\frac{m}{2}+1)}$.
\end{proof}

Now we are ready to prove the following result on the multiplicative order in $\Z[i]$, which plays a key role in the proof of Theorem \ref{thmfinite}; it may also be of independent interest. Lemma \ref{lemorderlowerEs} can be viewed as an extension of \cite[Lemma 2.5]{KLW25} to the complex case.

\begin{lem}\label{lemorderlowerEs}
Let \( \alpha \in \mathbb{Z}[i]\). Let \(\gamma, \eta_1,\xi_1,\ldots, \eta_q,\xi_q,\beta_1, \ldots, \beta_k\in\Z[i]\) be non-associate prime elements not dividing $\alpha$, where $\gamma$ is of Type II, $H(\eta_i)=H(\xi_i)$ for $1\leq i\leq q$, and $H(\beta_1),\ldots, H(\beta_k)$ are distinct. Then there exists a constant \( C_3 > 0 \) such that for any \( (h,r_1, s_1,\ldots, r_q,s_q,n_1, \ldots, n_k) \in \mathbb{N}_0^{k+2q+1}\), we have
\[
{\rm ord}(\alpha; \gamma^h\eta_1^{r_1}\xi_1^{s_1}\cdots\eta_q^{r_q}\xi_q^{s_q}\beta_1^{n_1}\cdots \beta_k^{n_k}) \geq C_32^{h/2}H(\eta_1)^{e_1}\cdots H(\eta_q)^{e_q}H(\beta_1)^{n_1}\cdots H(\beta_k)^{n_k},
\]
where $e_i=\max\{r_i,s_i\}$ for $1\leq i\leq q$.
\end{lem}

\begin{proof}
First, note that the assumption that the prime elements in Lemma \ref{lemorderlowerEs} are non-associate implies that 
\begin{equation}\label{EqDisprime}
2,\, H(\eta_1),\, \ldots \, H(\eta_q), \, H(\beta_1),\,\ldots, \, H(\beta_k)
\end{equation}
are distinct prime numbers. Moreover, since $H(\eta_i)=H(\xi_i)$ for $1\leq i\leq q$, we see that $\eta_i$ and $\xi_i$ are both of Type III. It then follows from the uniqueness part of Fermat's Two-Squares Theorem (cf. \cite[Theorem 3.66]{Rotman02A}) that $\eta_i$ is an associate of $\bar{\xi_i}$, the conjugate of $\xi_i$.

For convenience, write $\mathcal{B}=\{\eta_1,\xi_1,\ldots, \eta_q,\xi_q,\beta_1, \ldots, \beta_k\}$. Then by Lemma \ref{lemcompord}, for each $\tau\in\mathcal{B}$, there exist $m_{\tau},d_{\tau}\in \N$ such that 
 $$
{\rm ord}(\alpha; \tau) ={\rm ord}(\alpha;\tau^2) = \cdots ={\rm ord}(\alpha;\tau^{m_\tau})=d_\tau$$ and ${\rm ord}(\alpha;\tau^{m_{\tau}+j})=H(\tau)^jd_{\tau}$ for all $j\in\N.$
Hence we have for each $\tau\in\mathcal{B}$,
\begin{equation}\label{eqHbetambeta}
\quad H(\tau)^n \mid H(\tau)^{m_{\tau}} \cdot {\rm ord}(\alpha;\tau^n), \quad \forall n \in \mathbb{N}.
\end{equation}
Moreover, by Lemma \ref{Lemordbase2} there exists $C_2>0$ such that 
\begin{equation}\label{eqgamman}
{\rm ord}(\alpha;\gamma^n)\text{ is an integer power of } 2 \text{ and } {\rm ord}(\alpha;\gamma^n)\geq C_22^{\frac{n}{2}}, \quad \forall n\in\N.
\end{equation} 

Below we prove the lemma for all \( (h,r_1, s_1,\ldots, r_q,s_q,n_1, \ldots, n_k) \in \mathbb{N}^{k+2q+1}\); the other case that one of $h,r_1,s_1,\ldots, r_q,s_q,n_1,\ldots, n_k$ is $0$ can be proved similarly. Fix \( (h,r_1, s_1,\ldots, r_q,s_q,n_1, \ldots, n_k) \in \mathbb{N}^{k+2q+1}\). Write $$M={\rm ord}(\alpha; \gamma^h\eta_1^{r_1}\xi_1^{s_1}\cdots\eta_q^{r_q}\xi_q^{s_q}\beta_1^{n_1}\cdots \beta_k^{n_k}), \quad Q=\prod_{\tau\in\mathcal{B}}H(\tau)^{m_{\tau}}.$$
Then by \eqref{EqDisprime}-\eqref{eqHbetambeta}, the assumption that $H(\eta_i)=H(\xi_i)$ for $1\leq i\leq q$,  Lemmas~ \ref{lemdivord} and \ref{Lemordbase2}, we see that 
\[{\rm ord}(\alpha;\gamma^h)H(\eta_1)^{e_1}\cdots H(\eta_q)^{e_q}H(\beta_1)^{n_1}\cdots H(\beta_k)^{n_k}\mid QM,\]
where $e_i=\max\{r_i,s_i\}$ for $1\leq i\leq q$. It then follows that 
\begin{align*}
  M&\geq Q^{-1}{\rm ord}(\alpha;\gamma^h)H(\eta_1)^{e_1}\cdots H(\eta_q)^{e_q}H(\beta_1)^{n_1}\cdots H(\beta_k)^{n_k}\\
  &\geq Q^{-1}C_22^{h/2}H(\eta_1)^{e_1}\cdots H(\eta_q)^{e_q}H(\beta_1)^{n_1}\cdots H(\beta_k)^{n_k}.  \quad (\text{by }\eqref{eqgamman})
\end{align*}
Letting $C_3=Q^{-1}C_2$, we complete the proof.
\end{proof}

\section{Proof of Theorem \ref{thmfinite}}\label{S6}

Now we are ready to give the proof of Theorem~\ref{thmfinite}. The core of the proof is the application of Lemmas~\ref{lemperiod}-\ref{lemorderlowerEs}, together with an argument from \cite{KLW25}.

\begin{proof}[Proof of Theorem~\ref{thmfinite}]  We will prove the theorem when $\mathcal{S}$ consists of non-associate prime elements in $\Z[i]$ not dividing $\beta$ and is of the form 
\begin{equation}\label{EqSetS}
\mathcal{S}=\{\gamma,\eta_1,\xi_1,\ldots, \eta_q,\xi_q, \beta_1,\ldots,\beta_k\},
\end{equation}
where $\gamma$ is of Type II, $H(\eta_i)=H(\xi_i)$ for $1\leq i\leq q$,  and $H(\beta_1),\ldots, H(\beta_k)$ are distinct. For each of the other cases, we only need to replace $\mathcal{S}$ by a suitable subset of \eqref{EqSetS} and it will be apparent that a slight modification of the proof below yields the conclusion. Notice that since the prime elements in $\mathcal{S}$ are non-associate and $H(\eta_i)=H(\xi_i)$ for $1\leq i\leq q$, we infer from the classification of the prime elements in $\Z[i]$ (cf. Lemma \ref{LemClassify}) that for each $1\leq i\leq q$, $\eta_i$ and $\xi_i$ are both of Type III. Furthermore, by this and the uniqueness part of Fermat's Two-Square Theorem (cf. \cite[Theorem 3.66]{Rotman02A}), $\eta_i$ is an associate of the conjugate of $\xi_i$. 

Let $G_{\mathcal{S}}$ be the set of all Gaussian rational numbers whose denominators have only prime factors from $\mathcal{S}$. Our target is to show that $\#(G_{\mathcal{S}}\cap K_{\beta,D})<\infty$. We assume that $K_{\beta,D}$ is not a singleton; otherwise there is nothing to prove.  

We first show that if $\beta\in \Z$, then the theorem can be deduced from a known result on $\R$. Assume that $\beta\in\Z$. Then the IFS $\Phi_{\beta,D}^2:=\{\phi_i\circ\phi_j: 1\leq i,j\leq \ell\}$,  which still generates $K_{\beta,D}$, can be written in terms of affine transformations on $\R^2$ as 
\[\begin{pmatrix}
    \frac{1}{b} & 0 \\
    0 & \frac{1}{b} 
  \end{pmatrix}\begin{pmatrix}
                 x  \\
                 y
               \end{pmatrix}+\begin{pmatrix}
                               u_i\\v_i 
                             \end{pmatrix},\quad 1\leq i\leq m,\]
where $b=|\beta|^2$, $m=\#\Phi_{\beta,D}^2$ and $u_i,v_i\in \Q$ for $1\leq i\leq m$. Meanwhile, every element of $G_{\mathcal{S}}$, when written as a point in $\Q^2$, is of the form $(\frac{b_1}{a},\frac{b_2}{a})$, where $a$ is a product of powers of the norms $N(\vartheta)$, $\vartheta\in\mathcal{S}$. Note that by Lemma \ref{lemgcd}, $\gcd(b, N(\vartheta))=1$  for all $\vartheta\in\mathcal{S}$. Moreover, observe that if $(\frac{b_1}{a},\frac{b_2}{a})\in G_{\mathcal{S}}\cap K_{\beta,D}$, then $\frac{b_1}{a}$ lies in the self-similar set $K_1$ on $\R$ generated by the IFS $\left\{b^{-1}x+u_i\right\}_{i=1}^m$,
and $\frac{b_2}{a}$ lies in the self-similar set $K_2$ on $\R$ generated by the IFS $\left\{b^{-1}x+v_i\right\}_{i=1}^m,$ each of which has Hausdorff dimension less than $1$ as $K_1,K_2$ are orthogonal projections of $K_{\beta,D}$ to the coordinates. It then follows from  \cite[Theorem 1.1]{KLW25} that there are only finite many such points, yielding that $\#(G_{\mathcal{S}}\cap K_{\beta,D})<\infty$. Below we consider the case that $\beta\in \Z[i]\setminus \Z$.

Assume that $\beta\in \Z[i]\setminus \Z$. Then by Lemma \ref{lemnotline}, $K_{\beta,D}$ is not contained in a line in $\C$. Since $\Phi_{\beta,D}$ satisfies the WSP (cf. Lemma \ref{LemWSP}) and $\dim_{\rm H}K_{\beta,D}<1$, we can apply Lemma \ref{LemComposition} to find a composition $\Psi$ of $\Phi_{\beta,D}$ with similarity dimension less than $1$. Hence  by replacing $\Phi_{\beta,D}$ by $\Psi$ (which will not affect the result), we can assume that $\Phi_{\beta,D}$ has similarity dimension $s<1$.

For \( \mathbf{n} = (h,r_1,s_1,\ldots, r_q,s_q,n_1, \ldots, n_k) \in \mathbb{N}_0^{k+2q+1}\),  let $G_{\mathcal{S}}^{\mathbf{n}}$ be the set of all Gaussian rational numbers of the form
\[\frac{\alpha}{\gamma^h\eta_1^{r_1}\xi_1^{s_1}\cdots \eta_q^{r_q}\xi_q^{s_q}\beta_1^{n_1}\cdots \beta_k^{n_k}},\] 
where $\alpha\in \mathbb{Z}[i]$ with $\gcd(\alpha, \gamma^h\eta_1^{r_1}\xi_1^{s_1}\cdots \eta_q^{r_q}\xi_q^{s_q}\beta_1^{n_1} \beta_2^{n_2} \cdots \beta_k^{n_k})=1$.
Then $G_{\mathcal{S}}= \bigcup_{\mathbf{n} \in \mathbb{N}_0^{k+2q+1}}G_{\mathcal{S}}^{\mathbf{n}}.$
Note that for each \( \mathbf{n} \in \mathbb{N}_0^{k+2q+1}\), \( G_{\mathcal{S}}^{\mathbf{n}} \) is uniformly discrete; so we have \( \#(G_{\mathcal{S}}^{\mathbf{n}} \cap K_{\beta,D})< +\infty \). Therefore, to prove \( \#(G_{\mathcal{S}}\cap K_{\beta,D}) < +\infty \), it suffices to show that there exist only finite many   \( \mathbf{n}\in \mathbb{N}_0^{k+2q+1}\) such that $G_{\mathcal{S}}^{\mathbf{n}}\cap K_{\beta,D}\neq \emptyset$.

Let  \( \mathbf{n} = (h,r_1,s_1,\ldots, r_q,s_q,n_1, \ldots, n_k) \in \mathbb{N}_0^{k+2q+1}\) such that \( G_{\mathcal{S}}^{\mathbf{n}} \cap K_{\beta, D}\neq \emptyset \). Take
\[z=\frac{\alpha}{\gamma^h\eta_1^{r_1}\xi_1^{s_1}\cdots \eta_q^{r_q}\xi_q^{s_q}\beta_1^{n_1}\cdots \beta_k^{n_k}} \in G_{\mathcal{S}}^{\mathbf{n}}  \cap K_{\beta, D},\]
where \( \alpha \in \mathbb{Z}[i] \) with $\gcd(\alpha, \gamma^h\eta_1^{r_1}\xi_1^{s_1}\cdots \eta_q^{r_q}\xi_q^{s_q}\beta_1^{n_1} \beta_2^{n_2} \cdots \beta_k^{n_k})=1$. According to Lemma~\ref{lemperiod}, \( z \) has an eventually periodic coding \( j_1 j_2 \ldots j_m (j_{m+1} j_{m+2} \ldots j_{m+n})^\infty \in \Sigma^{\mathbb{N}} \) with period 
\begin{equation*}
n \leq C_1 H(\gamma^h\eta_1^{r_1}\xi_1^{s_1}\cdots \eta_q^{r_q}\xi_q^{s_q}\beta_1^{n_1}\cdots \beta_k^{n_k})^s.
\end{equation*}
Thus by Lemmas~\ref{LemHsubmul}, \ref{LemHalphaprime} and \ref{LemIIIpower}, we have 
\begin{align}
n&\leq C_1 \left(H(\gamma^h)H(\eta_1^{r_1}\xi_1^{s_1})\cdots H(\eta_q^{r_q}\xi_q^{s_q})H(\beta_1)^{n_1}\cdots H(\beta_k)^{n_k}\right)^s\nonumber\\
&\leq 2^sC_1 \left(2^{h/2}H(\eta_1)^{e_1}\cdots H(\eta_q)^{e_q}H(\beta_1)^{n_1}\cdots H(\beta_k)^{n_k}\right)^s,\label{eqpeoridup}
\end{align}
where $e_i=\max\{r_i,s_i\}$ for $1\leq i\leq q$.

Since \( D \subset \mathbb{Q}(i) \) is a finite set, we can find a non-zero \(\Theta \in \mathbb{Z}[i] \) so that \(\Theta t \in \mathbb{Z}[i] \) for all \( t \in D\).
Note that
\begin{align*}
z&= \sum_{i=1}^m \frac{t_{j_i}}{\beta^i} + \left(1 + \frac{1}{\beta^n} + \frac{1}{\beta^{2n}} + \cdots \right) \sum_{i=m+1}^{m+n} \frac{t_{j_i}}{\beta^i} \\
&= \sum_{i=1}^m \frac{t_{j_i}}{\beta^i} + \frac{\beta^n}{\beta^n-1} \sum_{i=m+1}^{m+n} \frac{t_{j_i}}{\beta^i},
\end{align*}
which yields that 
\[
z=\frac{\alpha}{\gamma^h\eta_1^{r_1}\xi_1^{s_1}\cdots \eta_q^{r_q}\xi_q^{s_q}\beta_1^{n_1}\cdots \beta_k^{n_k}}\in \frac{\mathbb{Z}[i]}{\beta^m (\beta^n - 1)\Theta}.
\]
Since \( \gcd(\alpha, \gamma^h\eta_1^{r_1}\xi_1^{s_1}\cdots \eta_q^{r_q}\xi_q^{s_q}\beta_1^{n_1}\cdots \beta_k^{n_k})=1\) and \( \gcd(\beta, \gamma^h\eta_1^{r_1}\xi_1^{s_1}\cdots \eta_q^{r_q}\xi_q^{s_q}\beta_1^{n_1}\cdots \beta_k^{n_k})=1\), it follows that
\[
\gamma^h\eta_1^{r_1}\xi_1^{s_1}\cdots \eta_q^{r_q}\xi_q^{s_q}\beta_1^{n_1}\cdots \beta_k^{n_k}\mid (\beta^n - 1)\Theta.
\]
Write \[\gcd(\Theta, \gamma^h\eta_1^{r_1}\xi_1^{s_1}\cdots \eta_q^{r_q}\xi_q^{s_q}\beta_1^{n_1}\cdots \beta_k^{n_k}) =\gamma^{h'}\eta_1^{r_1'}\xi_1^{s_1'}\cdots \eta_q^{r_q'}\xi_q^{s_q'}\beta_1^{n_1'}\cdots \beta_k^{n_k'}\] with \((h',r_1',s_1',\ldots, r_q',s_q',n_1', \ldots, n_k') \in \mathbb{N}_0^{k+2q+1}\). Then we obtain that
\[
\Upsilon:=\gamma^{h-h'}\eta_1^{r_1-r_1'}\xi_1^{s_1-s_1'}\cdots \eta_q^{r_q-r_q'}\xi_q^{s_q-s_q'}\beta_1^{n_1-n_1'}\cdots \beta_k^{n_k-n_k'}\cdots \beta_k^{n_k - n_k'} \mid \beta^n-1.
\]
That is, $\beta^n \equiv 1 \pmod{\Upsilon}$.
So we have 
\[{\rm ord}(\beta; \Upsilon) \mid n.
\]
Then, by Lemma~\ref{lemorderlowerEs}, we see that
\begin{align}
  n &\geq {\rm ord}(\beta; \Upsilon)\nonumber\\
  &\geq C_32^{\frac{h-h'}{2}}H(\eta_1)^{e_1'}\cdots H(\eta_q)^{e_q'}H(\beta_1)^{n_1-n_1'}\cdots H(\beta_k)^{n_k-n_k'}\nonumber\\
  &\geq C_32^{\frac{h-h'}{2}}H(\eta_1)^{e_1-e_1''}\cdots H(\eta_q)^{e_q-e_q''}H(\beta_1)^{n_1-n_1'}\cdots H(\beta_k)^{n_k-n_k'}\label{eq3inq}\\
  &= C_3R^{-1}2^{\frac{h}{2}}H(\eta_1)^{e_1}\cdots H(\eta_q)^{e_q}H(\beta_1)^{n_1}\cdots H(\beta_k)^{n_k},\label{eq4inq}
\end{align}
where $e_i'=\max\{r_i-r_i',s_i-s_i'\}$, $e_{i}''=\max\{r_i',s_i'\}$ for $1\leq i\leq q$, and 
\[R=2^{\frac{h'}{2}}H(\eta_1)^{e_1''}\cdots H(\eta_q)^{e_q''}H(\beta_1)^{n_1'}\cdots H(\beta_{k})^{n_k'}.\]
Note that in \eqref{eq3inq} we have used the fact that for  $1\leq i\leq q$, $$\max\{r_i-r_i',s_i-s_i'\}\geq \max\{r_i,s_i\}-\max\{r_i'-s_i'\}.$$ 
Now, combining \eqref{eqpeoridup} and \eqref{eq4inq} yields that 
\begin{equation}\label{FinalInq}
C_3R^{-1}A_{\mathbf{n}}\leq  2^sC_1A_{\mathbf{n}}^s,
\end{equation}
where $A_{\mathbf{n}}=2^{\frac{h}{2}}H(\eta_1)^{e_1}\cdots H(\eta_q)^{e_q}H(\beta_1)^{n_1}\cdots H(\beta_k)^{n_k}$. Since $s<1$ and each factor in $A_{\mathbf{n}}$ is at least $2$, the inequality \eqref{FinalInq} is satisfied by only finitely many $\mathbf{n}\in \N_0^{k+2q+1}$.
Hence $G_{\mathcal{S}}^{\mathbf{n}}\cap K_{\beta,D}\neq \emptyset$ only for finitely many $\mathbf{n}\in \N_0^{k+2q+1}$. This completes the proof of the theorem.
\end{proof}

\begin{rem}
We note that the first part of the above proof can be extended to the diagonal self-affine sets in $\R^2$. More precisely, let $F\subset\R^2$ be the self-affine set generated by the IFS
\[\begin{pmatrix}
    \frac{1}{a} & 0 \\
    0 & \frac{1}{b} 
  \end{pmatrix}
  \begin{pmatrix}
  x  \\y
  \end{pmatrix}+\begin{pmatrix}
  u_i\\v_i 
  \end{pmatrix},\quad 1\leq i\leq m,\]
where $a,b\in\Z$ with $|a|, |b|>1$, $u_i,v_i\in\Q$, $1\leq i\leq m$. We call $F$ a  diagonal self-affine set. Then a slight modification of the first part of the proof of Theorem~\ref{thmfinite} yields the following statement: Let $\mathcal{S}\subset \Z[i]$ be a finite set of non-associate prime elements with $\gcd(a,\vartheta)=\gcd(b,\vartheta)=1$ for all $\vartheta\in\mathcal{S}$. Suppose that the orthogonal projections of $F$ to the coordinates both have Hausdorff dimension less than $1$. Then we have $\#(F\cap G_{\mathcal{S}})<\infty$. 
\end{rem}

{\noindent \bf  Acknowledgements}. The author would like to thank De-Jun Feng, Bing Li, Jiangtao Li, Ruofan Li, Cai-Yun Ma and Zhao Shen for their helpful comments and suggestions.  This project is supported by the Natural Science Foundation of China (Grant No. 12301110) and Natural Science Foundation of Hunan Province (Grant No. 2023JJ40700). 


\end{document}